\documentclass{amsart}

\usepackage{amssymb}

\newcommand{\Ps}{\mathbf{P}}

\newcommand{\C}{\mathbf{C}}
\newcommand{\Q}{\mathbf{Q}}
\newcommand{\Z}{\mathbf{Z}}

\newtheorem{lemma}{Lemma}[section]
\newtheorem{proposition}[lemma]{Proposition}
\newtheorem{theorem}[lemma]{Theorem}
\newtheorem{corollary}[lemma]{Corollary}

\theoremstyle{definition}

\newtheorem{construction}[lemma]{Construction}

\theoremstyle{remark}
\newtheorem{remark}[lemma]{Remark}

\DeclareMathOperator{\codim}{codim}
\DeclareMathOperator{\rank}{rank}
\DeclareMathOperator{\rk}{rk}

\DeclareMathOperator{\Cl}{Cl}

\DeclareMathOperator{\Pic}{Pic}

\DeclareMathOperator{\sing}{sing}
\title{The Ciliberto-Di Gennaro conjecture for $d=5$}
\author{Remke Kloosterman}
\begin{document}
\begin{abstract}
The Ciliberto-Di Gennaro conjecture predicts that a nodal hypersurface of degree $d\geq 3$ with at most $2(d-2)(d-1)$  nodes is either factorial, or contains a plane and has at least $(d-1)^2$ nodes, or contains a quadric surface and has $2(d-2)(d-1)$ nodes. This conjecture is classically known for $d=3,4$. In 2022 the author proved  this conjecture for $d\geq 7$ by the author. Kvitko announced a proof for $d=6$ in 2025.
In this paper we prove the conjecture for the remaining open value of $d$, namely $d=5$.
\end{abstract}
\maketitle
\section{Introduction}
The Ciliberto-Di Gennaro conjecture predicts that a non-factorial threefold $X$ of degree $d\geq 3$ in $\Ps^4$ with at most $2(d-2)(d-1)$ nodes contains either a plane or a quadric surface. 
 This conjecture has been proven for $d=3$ by \cite{FinWer} and for $d=4$ by \cite{ChelPac,Shramov}. 
 We proved this conjecture for $d\geq 7$ in \cite{KloMaxNod}. The latter proof uses hyperplane sections to restrict the possibilities for the Hilbert function of the ideal of the nodes. There is a general argument which is enough for the case that $d\geq 8$  and $X$ has at most $2(d-2)(d-1)$ nodes to conclude that $X$ contains a plane or  a quadric surface. For $d=7$ this apporach is insufficient and yields one additional possibile Hilbert function for the ideal of nodes. We excluded this Hilbert function by an ad hoc argument. For $d=6$ there are several possibile Hilbert functions not excluded by our approach. Recently, Kvitko \cite{Kvitko} managed to exclude all of these and  thereby proving the conjecture for $d=6$. In this paper we will consider the remaining case $d=5$. We will show:

 \begin{theorem}
  Let $X\subset \Ps^4$ be a nodal hypersurface of degree $d=5$, Suppose that $X$ is non-factorial and that $X$ has at most $2(d-2)(d-1)=24$ nodes then
  either $X$ contains a plane and has at least $16$ nodes or $X$ contains a quadric surface and has exactly $24$ nodes.
 \end{theorem}
 
 Part of our approach is similar to \cite{KloMaxNod}.
 The main idea for the proof is to consider the ideal of the nodes $J$ and take a general linear form $\ell$. Let $J_H=J+\langle \ell \rangle$. Then $\C[x_0,\dots,x_4]/J_H$ is an algebra which is a finite dimensional $\C$-vector space. One has that $X$ is non-factorial if and only if $(\C[x_0,\dots,x_4]/J_H)_{2d-4}$ is  nonzero. This implies that there is an ideal $I$ containing $J_H$ and such that $\C[x_0,\dots,x_4]/I$ is an Artinian Gorenstein algebra of socle degree $2d-4=6$ and whose vector space dimension is at most the number of nodes of $X$.
 
 The papers \cite{KloMaxNod} and \cite{Kvitko} continue by studying $I_{d-4}$. In the case $d=5$ this yields very little information and we have to take a different approach. We consider directly the possible $h$-vectors for $S/I$ and conclude that $I$ is a complete intersection ideal of multidegree $(1,2,3,4)$ or $(1,1,4,4)$.  
 We then conclude from this that $X$ contains a quadric surface of a plane, using an argument from \cite{KloMaxNod}.

In Section~\ref{secGotzmann} we discuss a result by Gotzmann on Hilbert functions and Hilbert polynomials of ideals and prove two corollaries needed in our proof. In Section~\ref{secAG} we construct an Artinean Gorenstein algebra of socle degree $2d-4$, given a non-factorial nodal hypersurface of degree $d$. In Section~\ref{secPf} we prove the main result.

\section{Gotzmann's result}\label{secGotzmann}

Let $S=\C[x_0,\dots,x_n]$ and let $I\subset S$ be a homogeneous ideal. Let $h_I$ be the Hilbert function of $I$, i.e., $h_I(k)=\dim (S/I)_k$. Let $p_I(t)\in \Q[t]$ be the Hilbert polynomial, i.e.,  the polynomial such that $h_I(k)=p_I(k)$ for $k\in \Z$, $k$ sufficiently large. We will recall two results on limitations of the Hilbert function of $I$, in case $V(I)$ is non empty.

Let $d\geq 1$ be an integer.
Let $c:=h_I(d)$. We can write $c$ uniquely as
\[c= \sum_{i=1}^d \binom{i+\epsilon_i}{i}\]
with $\epsilon_d\geq \epsilon_{d-1}\geq ... \geq \epsilon_1\geq -1$. We call this the \emph{(Macaulay) expansion} of $c$ in base $d$.
This expansion can be obtained inductively as follows: The number $\epsilon_d$ is the largest integer such that $\binom{d+\epsilon_d}{d}\leq c$. The numbers $\epsilon_i$ for $i<d$ are the coefficients in the expansion of $c-\binom{d+\epsilon_d}{d}$ in base $d-1$.

Using the Macaulay expansion of $c$ we define the following  number $c^{\langle d \rangle}:= \sum_{i=1}^d\binom{i+\epsilon_i+1}{i+1}$.
Note that   $c\mapsto c^{\langle d \rangle}$  is an increasing function in $c$.

Recall the following theorem by Macaulay:
\begin{theorem}[{Macaulay \cite{Mac}}]\label{thmMac} Let $V\subset S_d$ be a linear system and $c=\codim V$. Then the codimension of $V\otimes_{\C} S_1$ in $S_{d+1}$ is at most $c^{\langle d\rangle}$.
\end{theorem}

The following result will be used to detect the Hilbert polynomial of the ideal generated by $I_d$:
\begin{theorem}[{Gotzmann \cite{Gotz}}]\label{thmGotz}
Let $V\subset S_d$ be a linear system and let $I\subset S$ be the ideal generated by $V$. Set $c=h_I(d)$.  If $h_I(d+1)=c^{\langle d \rangle}$ then for all $k\geq d$ we have $h_I(k+1)=h_I(k)^{\langle k \rangle}$. In particular, the Hilbert polynomial $p_I(t)$ of $I$ is given by
\[ \sum_{i=1}^d \binom{t+\epsilon_i}{t}\]
and the dimension of $V(I)$ equals $\epsilon_d$. 
\end{theorem}

Let $I$ be an ideal. Let $\eta_d$ be the value of $\epsilon_d$ in the Macaulay expansion of $h_I(d)$. 
Since $c\mapsto c^{\langle d \rangle}$ is increasing, we find that $d\mapsto \eta_d$ is a non increasing function. Hence $\eta_d\geq \dim V(I)$ for all $d\geq 1$.

\begin{corollary}\label{corBase} Let $I\subset S$ be an ideal, let $d$ be an integer let $m$ be the dimension of the base locus of $|I_d|$. Then $h_I(d)\geq h_{\Ps^m}(d)$.
\end{corollary}

\begin{proof} We may replace $I$ by the ideal generated by $\oplus_{j=0}^d I_j$, so that the base locus of $I_k$ is precisely $X=V(I)\subset \Ps^n$. Consider the  Macaulay expansion of $h_I(d)$ in base $d$
\[c= \sum_{i=1}^d \binom{i+\epsilon_i}{i}.\]
As mentioned above $\epsilon_d \geq \dim X=m$, hence 
\[ c\geq \binom{d+m}{d}=h_{\Ps^m}(d).\]
\end{proof}

If  $c\leq d$ then we have the following Macaulay expansions in base $d$: $\epsilon_d=\dots=\epsilon_{d-c+1}=0$ and $\epsilon_{d-c}=\dots=\epsilon_1=-1$. Hence $c^{\langle d\rangle}=c$. Combining this with Gotzmann's result yields:

\begin{corollary}\label{corGreen} Let $I\subset S$ be an ideal such that $c=h_I(d)\leq d$ and $h_I(d+1)=h_I(d)$. Then $I$ is the ideal of a zero-dimensional scheme of length $c$. In particular, $I_{d+1}$ is not base point free. 
\end{corollary}
\begin{proof} From Gotzmann's result it follows that the Hilbert polynomial of $I$ equals the constant $c$. Hence $V(I)$ is zero dimensional and of length $c$. 
\end{proof}

\section{The Artinian Gorenstein algebra associated to $X_{\sing}$}\label{secAG}
We start by summarizing some of the results mentioned in \cite[Section 3 and Proof of Theorem 4.1]{KloMaxNod}. The upshot of this is the construction of an Artinian Gorenstein algebra, associated to a non-factorial nodal  threefold $X\subset\Ps^4$, i.e., such that $\Cl(X)\neq \Pic(X)$. This algebra does not need to be unique.

Let $S=\C[x_0,x_1,x_2,x_3,x_4]$. Let $X=V(F)\subset \Ps^4$ be a nodal hypersurface of degree d. Suppose that $X$ is non-factorial, i.e., $\Pic(X) \subsetneq \Cl(X)$. It is well-known that then $\rank \Cl(X)>\rank \Pic(X)$ and that $h^4(X)>h^2(X)$ hold.
Let $J=I(X_{\sing})$. % For an ideal $I\subset S$ denote as in the previous  section with $h_I(k)=\dim (S/I)_k$ the Hilbert function of $I$. Let $p_I(k)$ be the Hilbert polynomial of $I$, the polynomial such that $p_I(k)=h_I(k)$ for large $k$.

The following result is well-known. For a discussion see for example \cite{KloMaxNod}.
\begin{proposition} Suppose $X$ is  nodal hypersurface in $\Ps^4$ of degree $d$. Then $X$ is factorial if and only if $X$ is $\Q$-factorial and
\[ h^4(X)-h^2(X)=\rk \Cl(X)-\rk \Pic(X)=p_J(2d-5)-h_J(2d-5).\]
In particular, if $X$ is not factorial then $p_J(2d-5)>h_J(2d-5)$.
\end{proposition}

Details of the following construction can be found in \cite[Proof of Theorem 4.1]{KloMaxNod}:
\begin{construction}
Suppose now that $X$ is non-factorial. Let $\ell$ be a general linear form and $R=S/\ell\cong \C[x_0,x_1,x_2,x_3]$. Consider the ideal $J_H$ of $R$ which is the image of $I+\langle \ell \rangle$ in $R$.
Since the partial derivatives of $F$ are contained in $J$, and a general hyperplane section of $X$ is smooth we get that $|J_{d-1}|$ is base point free.

The number of nodes of $X$ equals the constant polynomial $p_J$. Since $J$ is saturated and zero-dimensional and $H$ is general we obtain that
\[ p_J(2d-4)\geq h_J(2d-4)=\sum_{k=0}^{2d-4} h_{J_H}(k)\geq  \sum_{k=0}^{2d-4} h_I(k).\]
Moreover since $X$ is non-factorial we have that $p_J(2d-5)-h_J(2d-5)>0$ from which it easily follows that $h_{J_H}(2d-4)\neq 0$ (see \cite{KloMaxNod}). Hence $(J_H)_{2d-4}\neq R_{2d-4}$. 
 Using this and the fact that $(J_H)_{k}$ is base point free for $k\geq d-1$, one obtains that there is an ideal $I$ such that  $R/I$ is a quotient of $R/(J_H)$ and $R/I$ is an Artinian Gorenstein algebra of socle degree $2d-4$. The vector space dimension of $R/I$ is at most the number of nodes, which in the next section assume to be at most $2(d-2)(d-1)$.
 \end{construction}
\begin{remark}
Let $h=(a_0,\dots,a_{2d-4})$ be the $h$-vector of $S/I$, i.e., $a_k=h_I(k)=\dim_{\C} (S/I)_k$, for $k=0,\dots,2d-4$. We know that it is symmetric $a_k=a_{2d-4-k}$.
% and that  $a_0=1$ and that $a_i\geq a_{i-1}$ for $i=1,\dots,d-2$, see \cite{Kvitko}. Moreover, Kvitko uses that the difference $b_i=a_{i}-a_{i-1}$ for $i=1,\dots,d-2$ is an O-sequence, a fact which we will not need.
\end{remark}
\begin{lemma}\label{lemdi}
Let $d_j$ be the smallest degree $k$ such that the base locus of $|I_k|$ has dimension $j$. We have that $1\leq d_2\leq d_1\leq d_0\leq d_{-1}\leq d-1$.
Moreover, there exists polynomials $F_2,F_1,F_0,F_{-1}\in I$ such that $\deg(F_i)=d_i$ and such that the ideal $I_{CI}=(F_2,F_1,F_0,F_{-1})$ is  a complete intersection ideal. 

Then there exist a form $h$ of degree $d_2+d_1+d_0+d_{-1}-2d$ such that $I=(I_{CI}:h)$. In particular, $\sum_{i=-1}^2 d_i\geq 2d$ and if $\sum_{i=-1}^2 d_i=2d$ then $I=I_{CI}$.
\end{lemma}
\begin{proof}
The existence of $F_2,F_1,F_0,F_{-1}$ is obvious. Moreover, since $I_{d-1}$ is base point free, we find $d_{-1}\leq d-1$.

The algebra $R/I$ is a quotient of $R/I_{CI}$,  both algebras are Artinean Gorenstein algebras. Hence there is a form $h$ such that $I=(I_{CI}:h)$ with $\deg(h)$ the difference of the socle degrees of $R/I_{CI}$ and $R/I$.
\end{proof}

\begin{remark}The ideal $I_{CI}$ depends on some choices, but its $h$-vector is determined by $I$.
\end{remark}

\section{Proof of the Theorem}~\label{secPf}
In this section let $X=V(F)$ be  a degree 5 hypersurface in $\Ps^4$ with only nodes as singularities,  with defect and with at most $2(d-2)(d-1)=24$ nodes. Let $I\subset R$ be an ideal constructed form $I(X_{\sing})$ as in the previous section. Then $I$ has socle degree $2d-4=6$ and $I_{d-1}=I_4$ is base point free.

Let $(1,a_1,a_2,a_3,a_4,a_5,1)$ be the $h$-vector of $I$. Since $R/I$ is Artinian Gorenstein we have that  $a_4=a_2$ and $a_5=a_1$.
Moreover, since $a_1\leq 4$ by construction and $2+\sum a_i\leq 24$ we have that the $h$-vector is unimodal, i.e, $1\leq a_1\leq a_2\leq a_3$, see \cite{MNZ}.

The key part of our proof is the following proposition:

\begin{proposition}\label{prpIdealShape} The ideal $I$ is a complete intersection ideal of multidegree either $(1,1,4,4)$ or $(1,2,3,4)$.
\end{proposition}

The proof will follow from the following lemmas. 

\begin{lemma} We have that $a_1\in \{2,3\}$.
\end{lemma}

\begin{proof}
Since $R/I$ has socle degree $6$ and $I_1\subset R_1=\C[x_0,x_1,x_2,x_3]_1$ we obtain that $1\leq a_1\leq 4$.

Suppose $a_1=1$ then the base locus of $I_1$ is a point. Therefore $d_2=d_1=d_0=1$ and $d_{-1}\geq 10-d_2-d_1-d_0=7$, contradicting $d_{-1}\leq 4$.

Suppose now that $a_1=4$ and $a_2=4$.  Then $a_4=a_5=4$. From Corollary~\ref{corGreen} it follows that the base locus of $I_4$ and of $I_5$ is not empty, contradicting that $I_4$ is base-point-free.

Suppose now that $a_1=4$ and $a_2\geq 5$, then $a_3\geq 5$ and $2+\sum_{i=1}^5 a_i\geq 25$, contradicting the fact that this number is at most the number of nodes, which is at most 24.
\end{proof}

\begin{lemma} 
If $a_1=2$ then $I$ is a complete intersection ideal of multidegree $(1,1,4,4)$. 
\end{lemma}

\begin{proof} If $a_1=2$ then $d_2=d_1=1$. From $d_i\leq 4$ and $d_2+d_1+d_0+d_{-1}\geq 2d= 10$, it follows now that $d_0=d_{-1}=4$ and that $\sum d_i=10$. In particular, $S/I$ and $S/I_{CI}$ are Artinian Gorenstein algebras with the same socle degree, and since $I_{CI}\subset I$, it follows that $I=I_{CI}$.
\end{proof}

\begin{lemma} 
If $a_1=3$ then $a_2<6$. In particular, $d_2=1, d_1=2, d_0\in \{3,4\}$ and $d_{-1}=4$. 
Moreover, if $d_0=3$ then $I$ is a complete intersection ideal of multidegree $(1,2,3,4)$.
\end{lemma}

\begin{proof} 
If $a_5=a_1=3$ and $a_4=a_2\geq 6$ then $a_3\geq 6$. In particular $2+\sum_{i=1}^5 a_i\geq 26$, which contradicts the fact that this number is at most 24.
Hence $a_2<6$ and $d_1=2$. Using $\sum d_i\geq 10$  we obtain $d_{0}+d_1\geq 7$. From $d_0\leq d_{-1}\leq 4$ it follows that $d_{-1}=4, 3\leq d_0\leq 4$. If $d_0=3$ then both $I$ and $I_{CI}$ have the same socle degree and it follows that they coincide.
\end{proof}

\begin{lemma} 
If $a_1=3$ then  $d_0=3$. 
\end{lemma}

\begin{proof} 
By the previous lemmas we only have to exclude the case where $I$ satisfies $a_1=3, d_2=1,d_1=2,d_0=4=d_{-1}$. We will now derive a contradiction assuming that  $I$ has these invariants.
Let $I'=I_{CI}$ be a complete intersection ideal with multidegree $(d_2,d_1,d_0,d_{-1})$ contained in $I$. Then the socle degree of $S/I'$ equals $\sum d_i-4=7$, whereas the socle degree of $S/I$ equals 10. Hence there is a linear form $h\in S$  such that $I=(I':h)$.

Recall that  $I\subset \C[x_0,x_1,x_2,x_3]$  has a generator of degree $1$, and a generator of degree $2$. This yields $a_1=3,a_2=5$. If there is no further generator of degree at most $ 3$ then $I_3$ is the degree 3 part of the ideal of a plane conic $C$, and $a_3=7$. In particular, $2+\sum_{i=1}^5 a_i=25>24$, a contradiction. Hence $I$ has another generator of degree at most 3, but the base locus of $I_3$ is has the same dimension as the base locus of $I_2$. Hence the conic $C$ is reducible or non-reduced and the base locus of $I_3$ is a line.

After a change of coordinates we may assume that $x_3$ is the linear generator of both $I$ and $I'$. We mod out this generator and work with $\C[x_0,x_1,x_2]$ instead and the base loci are now $\Ps^2$. Then (the image of) $I'$ is generated by a quadric $Q$ and two forms of degree 4. The image of the  ideal $I$ contains the quadric $Q$ and at least one further generator of degree 2 or 3. The base locus in degree 3 has dimension one. However, hence the zerosets of these two generators have a common component which has to be a line, since it is a proper subset of the conic $V(Q)$. After a further change of coordinates we may assume that $x_0$ divides $Q$ and that the second generator of $I$ equals $x_0F$ where $F$ is either of degree 1 or of degree 2 and not a multiple of $Q/x_0$. Moreover, there are two further generators of $I'$ of degree $4$, say $G,H$, i.e., such that $I'=(Q,G,H)$.

Recall that $I=(I':h)$ for some linear form $h$.
Suppose first that $\deg(F)=2$. Since $x_0F\in I=(I':h)$, there exist a form $K$ of degree $2$, and constants $\lambda, \mu$ such that
\[ (x_0F)h= QK+\lambda F+\mu G.\]
Since $x_0$ divides $Q$, so does it divide $\lambda F+\mu G$.
Since $Q,F,G$ form a complete intersection, so do $x_0,F,G$. Hence the smallest degree to find a syzygy between  $x_0,F,G$ is $\deg(x_0)+\deg(F)$. From this it follows that $\lambda=\mu=0$.

In other words, independent of the value of $\deg(F)$ there exists a form $K$ such that
\[  x_0Fh=QK\]
Since $F$ is not a multiple of $Q/x_0$ we find that $h=Q/x_0$. However,  $Q\in I'$ and therefore $Q/h\in (I':h)=I$, a contradiction.
\end{proof}

These lemma's together prove Proposition~\ref{prpIdealShape}.

\begin{theorem} Suppose $X\subset \Ps^4$ is a non-factorial  nodal quintic threefold with at most 24 nodes. Then $X$ contains a plane and has at least 16 nodes or $X$ contains a quadric surface and has $24$ nodes.
\end{theorem}
This proof follows closely parts of the proofs of \cite[Lemma 5.3 and 5.4]{KloMaxNod}.
\begin{proof}
From Proposition~\ref{prpIdealShape} it follows that the ideal $I$ is a complete intersection ideal of multidegree $(1,1,4,4)$ or of multudegree $(1,2,3,4)$. 

Suppose first that $I$ is a complete intersection of multidegree $(1,1,4,4)$. Let $B$ be the base locus of $J_3$. Let $X=V(\ell)$. Then $B\cap H$ contains the base locus of $I_3$, which is a line. Since $H$ is general we obtain that $B$ has dimension at least 2. We can use the argument of \cite[Proof of Lemma 5.3]{KloMaxNod} to conclude that $B$ has dimension at most two: Suppose $B$ has dimension at least three, then from Corollary~\ref{corBase} we obtain that
\[ h_B(3)\geq h_{\Ps^3}(3)=20.\]
But then $h_I(6)\geq h_B(3)+h_I(4)+h_I(5)+h_I(6)=20+3+2+1=26>24$.
 Hence $\dim B=2$ and $B$ contains an irreducible component of dimension 2, such that the general hyperplane section of this component is a line, and therefore $B$ contains a plane.

As in the proof of \cite[Lemma 5.3]{KloMaxNod} it follows from Noether-Lefschetz theory  that the line contained in $B\cap H$ is contained in $X_H=X\cap H$, hence $X$ contains a plane.

If $I$ is a complete intersection of multidegree $(1,2,3,4)$ then the vector space dimension of $S/I$ equals $24$. Since it is a quotient of $S/(J_H)$, whose vector space dimension is at most 24 we find $I=J_H$.
But then $J$ is generated by a linear form $L$, a quadric $Q$ and generators of degree 3 and 4. 

Recall that $X$ is singular at $V(J)$, hence $F$ is contained in the saturation of $J^2$. As in \cite[Proof of Theorem 5.5]{KloMaxNod} we obtain that every generator of this saturated ideal of degree equal or less than 5 is divisible by $L$ or by $Q$. Hence  $F\in \langle L,Q\rangle$ and $X$ contains a quadric surface.

The number of nodes is at least $\sum_{k=0}^6 h_I(k)$ which equals 16 in the first case and 24 in the second case.
\end{proof}

\bibliographystyle{plain}
\bibliography{remke2}
\end{document}